\documentclass[12pt, article]{amsart}

\usepackage{tikz}
\usetikzlibrary{calc}

\usepackage{ae} % or {zefonts}
\usepackage[T1]{fontenc}
\usepackage[cp1250]{inputenc}
\usepackage{amsmath}
\usepackage{amssymb, amsfonts,amscd,verbatim}

\usepackage{indentfirst}
\usepackage{latexsym}
\input xy
\xyoption{all}

\usepackage{amsmath}    % need for subequations

\theoremstyle{plain}
\newtheorem{Prop}{Proposition}[section]

\newtheorem{theorem}[Prop]{Theorem}
\newtheorem{corollary}[Prop]{Corollary}

\newtheorem{lemma}[Prop]{Lemma}

\theoremstyle{definition}
\newtheorem{definition}[Prop]{Definition}

\theoremstyle{remark}
\newtheorem{remark}[Prop]{Remark}

\newcommand{\lec}{\operatorname{LEC}\nolimits}

\def\RR{{\mathbb R}}

\def\ZZ{{\mathbb Z}}

\def\UUU{{\mathcal U}}
\def\NNN{{\mathcal N}}
\def\III{{\mathcal I}}

\newcommand{\Rips}{\mathrm{Rips}}

\newcommand{\diam}{\operatorname{Diam}\nolimits}

\newcommand{\bx}{\operatorname{Box}\nolimits}
\newcommand{\rad}{\operatorname{Rad}\nolimits}

\newcommand{\St}{\operatorname{St}}
\newcommand{\Lk}{\operatorname{Lk}}
\newcommand{\Ind}{\operatorname{Ind}}

\numberwithin{equation}{section}
\title[Rips complexes of Integer lattices]
{Contractibility of the Rips complexes of Integer lattices via local domination}
\author{\v Ziga ~Virk}
\address{University of Ljubljana and Institute IMFM, Ljubljana, Slovenia}
\email{ziga.virk@fri.uni-lj.si}

\begin{document}

\maketitle
\begin{center}
\today
\end{center}

\begin{abstract}
We prove that for each positive integer $n$, the Rips complexes of the $n$-dimensional integer lattice in the $d_1$ metric (i.e., the Manhattan metric, also called the natural word metric in the Cayley graph) are contractible at scales above $n^2(2n-1)$, with the bounds arising from the Jung constants.  We introduce a new concept of locally dominated vertices in a simplicial complex, upon which our proof strategy is based. This allows us to deduce the contractibility of the Rips complexes from a local geometric condition called local crushing. In the case of the integer lattices in dimension $n$ and a fixed scale $r$, this condition entails the comparison of finitely many distances to conclude that the corresponding Rips complex is contractible. In particular, we are able to verify that for $n=1,2,3$, the Rips complex of the $n$-dimensional integer lattice at scale greater or equal to $n$ is contractible. We conjecture that the same proof strategy can be used to extend this result to all dimensions $n$.
\end{abstract}

%---------------------------------------------------------------------------
%%%%%%%%%%%%%%%%%%%%%%%%%%%%%
\section{Introduction}

Given a metric space $X$ with distance $d$ and $r \geq 0$, the (closed) \textbf{Rips complex} $\Rips (X, r)$ is the abstract simplicial complex with the vertex set $X$, for which a finite $\sigma \subseteq X$ is a simplex iff $\diam(\sigma) \leq r$. Here, $\diam(\sigma) = \max_{x,y\in \sigma} d(x,y).$

The Rips complexes, sometimes referred to as Vietoris-Rips complexes, were originally introduced by Vietoris \cite{Viet}, \cite[p. 271]{Lef} (following \cite{ZV3} we reserve the name ``Vietoris complex'' for Vietoris's cover based construction of a simplicial complex). The construction was later rediscovered by Rips in the context of geometric group theory and used by Gromov \cite{GR}. Rips proved that the Rips complexes of hyperbolic groups are contractible for large scales $r$ \cite[Proposition III.$\Gamma$.3.23]{Bridson}, when considering a group equipped with the word metric, i.e., as a  Cayley graph. This turns out to be an important aspect in terms of group actions and has far reaching ramifications in geometric group theory. There are many variations and generalizations of this result. Some of these include \cite{BauerRoll} (from the computational perspective) and \cite{Zar1, Zar2}, which also incorporate Bestvina-Brady Morse theory. The motivating question for our work has been posed by an author of the latter two papers, Matthew Zaremsky: Are the Rips complexes of the free finitely generated Abelian groups (integer lattices in $d_1$ metric) contractible for large scales? Our main result confirms that this is indeed the case, and thus represents an extension of Rips's result to the ``highly'' non-hyperbolic free finitely generated Abelian groups.

While the aspect of geometric group theory is one of the main motivations to study the above question, it is not the only one. The second context in which our results play a prominent role is that of the reconstruction results, i.e., the question whether the Rips complex of a space (or a nearby sample in, say, Gromov-Hausdorff metric) at a reasonably small scale attains the homotopy type of the space itself. This question was first studied by Hausmann in \cite{Haus}, where he has used the Rips complexes to define a new cohomology theory. He proved that the reconstruction result holds for a large class of Riemannian manifolds. A simpler proof of his reconstruction result using Vietoris complexes and Dowker duality was recently given in \cite{ZV3}. Motivated by the followup question of Hausmann, Latschev later proved the reconstruction result for a Gromov-Hausdorff nearby sample for a large class of Riemannian manifolds. His result has later been generalized and extended, sometimes within the setting of computational geometry, in \cite{Att, Lem, Sush2, Sush1, BauerRoll}, with the results of \cite{Zar1, Zar2} also fitting into this setting despite the geometric group theoretic framework. Unfortunately, none of the tools used in these papers turned out to applicable in our setting of integer lattices. By scaling the integer lattice to a sufficiently dense set, our main result is a specific variant of a reconstruction results. For an approach in dimensions $2$ and $3$ using hyperconvex embeddings see \cite{QW}.

The third context in which our result can be positioned, is that of persistent homology \cite{EH}, within which the Rips complexes play a fundamental role in applications \cite{Bauer, MC}. One of its fundamental properties is stability, a version of which also holds for the Rips complexes: If metric spaces are $\delta$-close in the Gromov-Hausdorff distance, then the persistence diagrams (obtained via the Rips filtrations) are at distance at most $2\delta$, \cite{Cha2}. It has led, for example, to a counterexample of Hausmann conjecture from \cite{Haus} on the monotonicity of the connectivity of Rips complexes, see \cite{ZVCounterex}. The stability theorem implies that the persistent homology of the Rips complexes of integer lattices is close to that of the Euclidean space in $d_1$ metric (which has trivial persistent homology), but not necessarily the same. Our main result implies that persistent homology (and the homotopy type of Rips complexes) of integer lattices are more than stable, they are rigid at large scales: at any sufficiently large scale they reconstruct the homology (and homotopy type) of the Rips complex of the Euclidean space at that scale (which is homotopically trivial). In contrast to the reconstruction results, which are a specific type of rigidity results holding for reasonably small scales, this property holds for sufficiently large scales. Outside the realm of the reconstruction results, the rigidity of persistent homology (or the homotopy type) has only been observed in a few settings: for flag complexes of circular graphs, including the homotopy type of Rips complexes of $S^1$ and certain ellipses \cite{AA, Ad5}, for $1$-dimensional persistence of geodesic spaces \cite{ZV1}, and for connectivity of the Rips complexes of spheres above small and medium scales \cite{ABV}.

\textbf{Main results.} Our main results state that the Rips complexes of the integer lattices in the $d_1$ metric are contractible for large scales. For dimensions $n=1, 2, 3$ we obtain the optimal (see Section \ref{Sect:Conclusion} for a discussion and a conjecture on the optimality) contractibility bound $r \geq n$. For larger dimensions, we obtain a bound arising from the Jung constants.
\begin{itemize}
 \item Theorems \ref{Thm:Rips12} and \ref{Thm:Main2}: Fix $n\in \{1,2,3\}$. Then for  any $r \geq n$,  $\Rips(\ZZ^n, r)$ is contractible, when using the $d_1$ (Manhattan) metric.
  \item Theorem \ref{Thm:Main1}: Fix a positive integer $n$.  Then for  any $r \geq  n^2 (2n-1)$, $\Rips(\ZZ^n, r)$ is contractible, when using the $d_1$ (Manhattan) metric. 
\end{itemize}

\textbf{A comment on other $d_p$ metrics.} For a definition of $d_p$ metrics for $p\in [1,\infty]$ see the beginning of Section \ref{Sect:Conclusion}. For example, $d_2$ is the Euclidean metric. In \cite[Example 4.7]{Zar1}, Matt Zaremsky has proved that $\Rips\left((\ZZ^n, d_2), t\right)$ is contractible for all $t > \sqrt{n}/(2 - \sqrt{3})$ using the Bestvina-Brady discrete Morse theory. He has later informed us that a similar argument also seems to work for $d_{1 + \varepsilon}$ metrics for $\varepsilon > 0$, although the argument hasn't been written down by the time of the writing. For this reason, the present paper focuses on the $d_1$ metric, and only comments on why the same arguments hold in other $d_p$ metrics in Section \ref{Sect:Conclusion}.

\textbf{Main tool.}
As our main tool we introduce a new concept: local domination of a vertex in a simplicial complex, see Definition \ref{Def:LocalDom}. Local domination is a generalization of the classical domination and allows us to remove a locally dominated (but not necessarily dominated) vertex from a simplicial complex without changing the homotopy type. As a novel tool in combinatorial topology it is of independent interest. While domination suffices to prove our main results in dimensions $1$ and $2$, higher dimensional lattices require the use of local domination. The metric counterpart of local domination which allows for a transition to Rips complexes is called a local crushing. It is a generalization of Hausmann's crushing \cite{Haus}.

The structure of the paper is as follows. 
\begin{description}
\item[Section \ref{Sect:Prelim}] we provide preliminaries.  
\item[Section \ref{Sect:LocDomCrush}] we demonstrate how domination is used to prove our first main result about contractibility in dimensions $1$ and $2$, see Theorem \ref{Thm:Rips12}. We then discuss why the approach fails in higher dimensions (Remark \ref{Rem:Rips12}) and introduce local domination (Definition \ref{Def:LocalDom} and Theorem \ref{Thm:LocalDom}), along with its metric counterpart aimed at Rips complexes, local crushing (Definition \ref{Def:LocCrush} and Theorem \ref{Thm:Crush}).
\item[Section \ref{Sect:LEC}] We introduce a variant of the local crushing in our context, the Local Euclidean crushing property $\lec$ (Definition \ref{Def:LEC}, Theorem \ref{Thm:JungBounds}). 
\item[Section \ref{Sect:MainContract}] We prove our main results on the contractibility of the Rips complexes of lattices in dimension $3$ (Theorem \ref{Thm:Main2}) and higher (Theorem \ref{Thm:Main1}) using $\lec$.
\item[Section \ref{Sect:Conclusion}] We discuss why our approach works for all $d_p$ metrics. We conjecture that our proof strategy can be used to prove optimal contractibility bounds, and comment on why $\Rips(\ZZ^n,r)$ is not contractible for $r<n$.
\end{description}

%---------------------------------------------------------------------------
%%%%%%%%%%%%%%%%%%%%%%%%%%%%%
\section{Preliminaries}
\label{Sect:Prelim}

Let $n$ be a non-negative integer. The point $(0, 0, \ldots, 0)\in \RR^n$ will be denoted by $0^n$. Throughout the paper, the distance $d$ on $\RR^n$ will always be the $d_1$ distance (also called the Manhattan distance) unless stated otherwise:
$$
d_1\left( (a_1, a_2, \ldots, a_n), (b_1, b_2, \ldots, b_n) \right) = \sum_{i=1}^n |a_i - b_i|.
$$
The one point space is denoted by $\bullet$. The notation $ X \simeq \bullet$ means $X$ is contractible. 
An (abstract) simplicial complex $K$ is a subset-closed collection of non-empty subsets of its vertex set. The relation of being a subcomplex is denoted by $\leq$. The $n^{\textrm{th}}$ skeleton of $K$ will be denoted by $K^{(n)} \leq K$. The vertex set of $K$ is $V(K)= K^{(0)}$. Given a vertex $a \in V(K)$, its (closed) star is $\St_K(a)=\{\sigma \in K \mid \{a\} \cup \sigma \in K \}\leq K$ and its link is $\Lk_K(a)=\{\sigma \in \St_K(a) \mid a\notin \sigma\} \leq \St_K (a)$. A simplicial complex $K$ is a \textbf{flag complex} if the following condition holds: $\sigma \subseteq V(K)$ is a simplex in $K$ iff each pair contained in $\sigma$ is a simplex in $K$. Given $A \subseteq V(K)$, the induced subcomplex is defined as $\Ind_K (A)=\{\sigma \in K \mid \sigma \subseteq A\} \leq K$. Given a subcomplex $K' \leq K$, the induced subcomplex is $\Ind_{K}(K') = \Ind_K(V(K')) \geq K'.$ Clearly, each Rips complex is a flag complex.

\begin{definition}
\label{Def:Dom}
 Let $K$ be a simplicial complex and $a,b\in V(K)$. A vertex $a$ is \textbf{dominated} by the vertex $b \neq a$ if for each $\sigma \in K: a\in \sigma \implies \{b\} \cup \sigma \in K$.
\end{definition}

%---------------------------------------------------------------------------
%%%%%%%%%%%%%%%%%%%%%%%%%%%%%
\section{Local domination and local crushing}
\label{Sect:LocDomCrush}

\begin{definition}
\label{Def:LocalDom}
 Let $K$ be a simplicial complex and $a\in V(K)$ a vertex. A vertex $a$ is \textbf{locally dominated} in $K$ if there exists a simplex $L_a \in K$ not containing $a$, such that the following condition holds: 
 for each $\sigma \in K$ with $a \in \sigma$, there exists $b_\sigma \in L_a$, such that $ \sigma \cup \{b_\sigma \} \in K$. In particular, $a$ is not isolated.
\end{definition}

Intuitively, $a$ is locally dominated if there exists a simplex $L_a$ in $K$ such that every simplex in the star of $a$ is joinable to a vertex of $L_a$.
Local domination is a combinatorially-local variant of the concept of domination. If $L_a$ can be chosen to be a vertex in Definition \ref{Def:LocalDom}, the corresponding local domination is the standard domination,  see Figure \ref{Fig3}.

\begin{figure}[htbp]
\begin{center}
\includegraphics[scale=1.2]{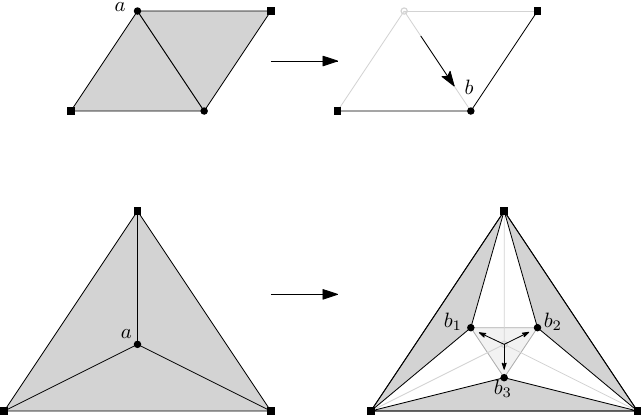}
\caption{Domination on top (Definition \ref{Def:Dom}) and local domination at the bottom (Definition \ref{Def:LocalDom}). 
Vertex $a$ in the top row being dominated means there exists a vertex $b$ contained in each maximal simplex containing $a$. Thus ``pushing'' $a$ to $b$ is a homotopy equivalence. 
Vertex $a$ in the bottom row being locally dominated means we can choose a designated $b_i$ in each maximal simplex containing $a$ (in our figure, the maximal simplices are the three shaded tetrahedra with the points $b_i$ hidden behind the front faces on the left side), so that the collection of the vertices $b_i$ forms a simplex. Theorem \ref{Thm:LocalDom} shows that, roughly speaking, ``spreading'' $a$ towards all $b_i$ is a homotopy equivalence, i.e., that the non-shaded area on the right can be filled in by some simplices. }
\label{Fig3}
\end{center}
\end{figure}

The following results will motivate local domination. It turns out that domination is sufficient to prove contractibility of Rips complexes of lattices in dimension $1$ and $2$ (Theorem \ref{Thm:Rips12}), but falls short in higher dimensions (Remark \ref{Rem:Rips12}).

\begin{theorem}
 \label{Thm:Rips12}
Let $n \in \{1,2\}$. Then for each $r \geq n$, $\Rips(\ZZ^n,r) \simeq \bullet$.
\end{theorem}

\begin{proof}
 By the Whitehead Theorem it suffices to show that all the homotopy groups are trivial. As any homotopy class is contained in the Rips complex on finitely many points, it suffices to show that for each positive integer $M $, $\Rips(Z^n_M, r) \simeq \bullet$, where $Z_M = \{0, 1,2,  \ldots, M\}$. 
 
 \textbf{n=1}: First, note that $0$ is dominated by $1$ in $Z_M$ as $r \geq 1$. Removing $0$ thus preserves the homotopy type. We proceed by induction (see Figure \ref{Fig1}) removing the left-most point until we reach the one-vertex simplicial complex $\Rips(\{M\},r)$.
 
 \textbf{n=2}: The proof is similar to the case $n=1$  in the sense that we are inductively removing dominated points from $Z^2_M$ in a certain order. Starting in the lowest row, we are removing the vertices from left to the right (see Figure \ref{Fig2}). Assume that $v$ is the leftmost vertex in the lowest row of the (potentially reduced) $Z^2_M$, and let the $y$ coordinate of $v$ be below $M$:
\begin{enumerate} 
 \item If $v$ is not the last vertex in the row, then $v$ is dominated by the vertex  $v + (1,1)$, and can thus be removed. Observe that the $d_1$ distance is the same as the induced path distance on the indicated grid on Figure \ref{Fig2}: at any scale, the shortest path $\alpha$ from a point $w\neq v$ to $v$ passes either through $v + (1,0)$ or through $v + (0,1)$ just before reaching $v$, and so we obtain a  path of the same or shorter length from $w$ to $v + (1,1)$ by diverting the last segment of $\alpha$ towards $v + (1,1)$.
 \item If a vertex $v$ is the last vertex in the row, then $v$ is dominated by the vertex  $v + (0,1)$, and can thus be removed. 
\end{enumerate}
 When the reduction reaches $Z_M \times \{M\}$ we refer to the case $n=1$ to conclude the contractibility.

Assumption $r \geq 2$ was used in two places. First, to assure that the pairs of vertices $(v, v + (1,1))$ and $(v, v + (0,1))$ form an edge. Second, to invoke case $n=1$, where $r \geq 1$ is required.
\end{proof}

\begin{figure}[htbp]
\begin{center}
\includegraphics[scale=.9]{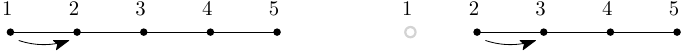}
\caption{The proof of Theorem \ref{Thm:Rips12} for $n=1$: we inductively remove the leftmost vertex, which is dominated by the second leftmost vertex.}
\label{Fig1}
\end{center}
\end{figure}

\begin{figure}[htbp]
\begin{center}
\includegraphics[scale=1.1]{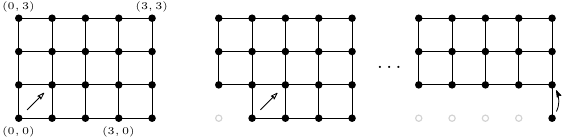}
\caption{The proof of Theorem \ref{Thm:Rips12} for $n=2$: we inductively remove the lexicographically smallest vertex, which is dominated by the vertex indicated by the arrow.}
\label{Fig2}
\end{center}
\end{figure}

\begin{remark}
 \label{Rem:Rips12}
 The vertex $0^3\in Z=\{0,1,2,3\}^3$ is not dominated in $\Rips(Z,3)$. While each of the three vertices $(3,0,0), (0,3,0), (0,0,3)$ forms an edge with $0^3$ (i.e., are at distance at most $3$), it is easy to see that there is no vertex on $Z$, other than $0^3$, that would be at distance at most $3$ from all three of them. Thus, domination cannot be used and we turn our attention to the more general local domination.
\end{remark}

The Theorem \ref{Thm:LocalDom} establishes the crucial property that removing a locally dominated (and in a special case, a dominated) vertex is a homotopy equivalence. 

\begin{theorem} [Local domination]
\label{Thm:LocalDom}
Let $K$ be a simplicial complex and $a\in V(K)$ a locally dominated vertex in $K$. If $K$ is a flag complex, then $\Ind_K(V(K) \setminus \{a\}) \simeq K$.
\end{theorem}

In the proof of Theorem \ref{Thm:LocalDom} we will use the following well known lemma, see for example \cite[Gluing Lemma 10.3]{Bjor}, or recent uses in \cite{Ziqin, Anu}.
%Lemma 2 of https://arxiv.org/pdf/2305.07084.pdf
%Lemma 2.6 of https://arxiv.org/pdf/2303.08798.pdf
% Lemma 2.1 of https://arxiv.org/pdf/2106.09915.pdf
%Bjoerner: Topological methods, Gluing Lemma 10.3

\begin{lemma}
 \label{Lem:LocalDom}
 Let $K$ be a flag simplicial complex and $a\in V(K)$. If $\Lk_K(a) \simeq \bullet$, then $K \simeq \Ind_K(V(K) \setminus \{a\})$.
\end{lemma}

\begin{proof}
[Proof of Theorem \ref{Thm:LocalDom}]
 According to Lemma \ref{Lem:LocalDom} it suffices to show that $\Lk_K(a) \simeq \bullet$. Choose a simplex $L_a$ according to Definition \ref{Def:LocalDom}. Define a cover $\UUU$ of $\Lk_K(a)$ by subcomplexes $\UUU = \{M_\sigma\}_{\sigma \in \Lk_K(a)}$, where $M_\sigma = \Ind_{\Lk_K(a)}( \sigma \cup L_a)$. Observe that the nerve $\NNN(\UUU)$ is the full simplex as the intersection of $\UUU$ is non-trivial (it contains at least $L_a$). In particular, $\NNN(\UUU) \simeq \bullet$. By the nerve theorem it remains to prove that $\UUU$ is a good cover of $\Lk_K(a)$.
 
 Let $\sigma_1, \sigma_2, \ldots, \sigma_\ell$ be a collection of simplices in $\Lk_K(a)$. We intend to prove that $\III = \bigcap_{i=1}^\ell M_{\sigma_i} \simeq \bullet$. Observe that $\III = \Ind_{L_a}\left(L_a \cup  \bigcap_{i=1}^\ell \sigma_i\right)$. 
 In a special case when $\bigcap_{i=1}^\ell \sigma_i = \emptyset$ we have $\III = L_a$ is a simplex thus clearly contractible.
 In general, $\bigcap_{i=1}^\ell \sigma_i \in \Lk_K(a)$ and by Definition \ref{Def:LocalDom}, there exists $b\in L_a$, such that $\bigcap_{i=1}^\ell \sigma_i  \cup \{b\} \in \Lk_K(a)$. This implies that every vertex of 
 $\bigcap_{i=1}^\ell \sigma_i$ is connected to $b$.
 As $\III$ is a flag complex and $b \in L_a$, this implies $\III$ is a cone with the apex at $b$ (as every vertex of $\III$ is connected to $b$) and thus contractible. As a result, $\UUU$ is a good cover and the theorem follows.
\end{proof}

We now translate local domination concept into a metric setting amenable to Rips complexes. 

\begin{definition}
\label{Def:LocCrush}
 Let $(X, d)$ be a finite metric space, $x \in X$, and $r>0$. Point $x$ is \textbf{locally crushable} (at scale $r$) in $X$ if there exists a subset $L_x \subset X$ not containing $x$ and of diameter $\diam(L_x) \leq r$, such that for each $A \subset X$ satisfying $x\in A, \diam(A) \leq r$, there exists $y\in L_x$: $\diam(A \cup \{y\}) \leq r$.
\end{definition}

The concept of a crushing as a method of transforming the underlying set of a Rips complex while preserving the homotopy type of the complex has first been used by Hausmann \cite{Haus}. Crushing and its discrete variant have later been used in \cite{Lat,  ZV2, AA, Rolle, Lem}. Definition \ref{Def:LocCrush} provides a combinatorially-local version of crushing. 

The following theorem is a direct consequence of Theorem \ref{Thm:LocalDom} and the definition of Rips complexes. 

\begin{theorem}
 \label{Thm:Crush}
 Let $X$ be a finite metric space, $x \in X$, and $r>0$. If point $x$ is locally crushable in $X$, then $\Rips(X,r) \simeq \Rips(X \setminus \{x\},r)$.
\end{theorem}

%---------------------------------------------------------------------------
%%%%%%%%%%%%%%%%%%%%%%%%%%%%%
\section{Local Euclidean crushing}
\label{Sect:LEC}

In this section we prove a variant of the local crushing property for compact subsets of a Euclidean space, see Definition \ref{Def:LEC}. While our intended application is to finite subsets (i.e., the simplices of Rips complexes), the use of the Jung constant allows us to phrase it for compact subsets. Throughout this section we fix a positive integer $n$.

Given a compact subset $A \subset \RR^n$, let  $\bx(A)$ denote the smallest box of the form $\prod_{i=1}^n [a_i, b_i] \subset \RR^n$ containing $A$.  In particular, each $a_i$ is the minimal value of the $i^{\textrm{th}}$ coordinate attained by $A$. 
%Define $D(n,p) = \diam_p([-1,1]^{n-1} \times [0,1])$. Note that $D(n, \infty) = 2$ and $D(n,p)= (1 + (n-1)2^p)^{1/p}$ for finite $p$.

\begin{definition}
 \label{Def:LEC}
Given a positive integer $n$  and $\rho >0$, we say that the \textbf{local Euclidean crushing property} $\lec(n, \rho)$ holds, if for each compact $\tau \subset \RR^{n-1} \times [0, \infty)$ with $0^n \in \tau, \diam(\tau) \leq 2n-1$, there exists $c_\tau \in ([-1,1]^{n-1} \times [0,1]) \cap \bx(\tau)$, such that the following holds: for each $x\in \tau, d(x, c_\tau) \leq (2n-1) - \rho$.
\end{definition}

Observe that if $\lec(n, \rho)$ holds, then $\lec(n, \rho')$ also holds for all positive $\rho' < \rho.$

\begin{remark}
\label{Rem:Box}
Here we provide a few comments on Definition \ref{Def:LEC}. 
In a forthcoming argument, Definition \ref{Def:LEC} will be used inductively in the lexicographical order. To that purpose, we have phrased it in terms of subsets of $\RR^{n-1} \times [0, \infty)$ (where all lexicographical successors of $0^n$ lie) instead of $\RR^n$.
The definition is essentially phrasing a local crushing condition for compact subsets in $\RR^{n-1} \times [0, \infty)$ with the redundancy $\rho$. The redundancy will later be required to pass to the lattices. Without the loss of generality, the ``anchor'' point is taken to be $0^n$. The role of the subset $L_x$ from Definition \ref{Def:LocCrush}, onto which the local crushing is occurring, is assigned to the box $[-1,1]^{n-1} \times [0,1]$ as one of the simplest forms of easily determined diameter. Its diameter $2n-1$ is thus used as the diameter bound. 

The proof of the variant of the $\lec$ property using $\RR^n$ instead of $\RR^{n-1} \times [0, \infty)$, and $L_x=[-1,1]^n$ instead of $[-1,1]^{n-1} \times [0,1]$,  is identical to the given proof. 
\end{remark}

We will now explain how the constant $\rho$ is related to the Jung constant. 

Given $x\in \RR^n$ and $r>0$, define the closed ball $B(x,r)= \{y\in \RR^n \mid d_1(x,y) \leq r\}$.
For a compact subset $A \subset \RR^n$ define the smallest enclosing radius $\rad(A)$ as the minimal radius of a ball in $\RR^n$ containing $A$, and keep in mind that we are using the $d_1$ metric. By compactness, the minimum $\rad(A)$ exists. The center $C(A)$ of such a ball may not be unique though. For $A'= \{(1,0),(0,1)\}$, any point on the closed line segment from $(0,0)$ to $(1,1)$ is the center of such a ball. Therefore, a center may not be contained in the convex hull of $A$ (where by convex hull we mean the set of all convex combinations of points of $A \subset \RR^n$, not the geodesic convex hull). However, each center is in $\bx(A)$. This is easy to see: if a point $C'$ has, for example, the first coordinate smaller than the first coordinates of all points of $A$, then increasing its first coordinate to the minimal first coordinate of $A$ decreases the distances from $C'$ to each point of $A$. Performing such a modification for each coordinate we obtain a $C(A)$ within $\bx(A)$.

\begin{definition}
 \label{Def:Jung}
Given a positive integer $n$,  the \textbf{Jung constant} is defined as
$$
J(n) = \sup \{\rad(A) \mid A \subset \RR^n, \diam(A) \leq 1\}.
$$
\end{definition}

There is a long history of results on parameters $J(n)$, starting with \cite{Jung} who proved that $J(n, d_2) =\sqrt{\frac{n}{n+1}}$ holds for the Euclidean metric $d_2$. For a review of the topic see \cite{Varhatis} or \cite{Amir}. For our argument it will be crucial to have $J(n) < 1$ (which holds by (1) of Theorem \ref{Thm:JungBounds}), while explicit upper bounds will provide eventual bounds on the scales at which the Rips complexes of integer lattices are contractible.  The following proposition states some of the bounds on the Jung constant that we will be using. 

\begin{theorem}
 \label{Thm:JungBounds}
The following are some of the bounds on the Jung constant in the $d_1$ metric:
\begin{enumerate}
 \item The bound $J(n) \leq \frac{n}{n+1}$ was proved by \cite{Bohn}, see also a short proof using the Helly theorem in \cite[Proposition 2.12]{Amir}
 \item $J(n) = \frac{n}{n+1}$ iff there exists a Hadamard matrix of order $n + 1$. This result was proved by \cite{Dol}, see also \cite[Section 15]{Alimov} and \cite{Ivanov} for some related results.
\end{enumerate}
\end{theorem}

In general, determining the exact value of the Jung constant for $\RR^n$ with the $d_1$ metric is an open problem.

\begin{lemma}
[$\lec$ and the Jung constant]
 \label{Lem:Main}
 Fix an integer $n > 1$ and $\kappa \in [J(n),1)$. Then $\lec(n,\rho)$ holds for $\rho = \frac{1-\kappa}{\kappa} $.
\end{lemma}

\begin{proof}
 Choose a compact $\tau \subset \RR^{n-1} \times [0, \infty)$ with $0^n \in \tau, \diam(\tau) \leq 2n-1$. 
 By the discussion above we may choose the center $c'_\tau$ of a minimal ball containing $\tau$ within $\bx(\tau)$, i.e., $c'_\tau \in \bx(\tau)$. Using the scaling of the Jung constant we obtain $\rad(\tau) \leq \diam(\tau) J(n)\leq  (2n-1)\kappa$. In particular, for each $ x\in \tau: d(x, c'_\tau) \leq (2n-1) \kappa.$ 
 Observe that $(2n-1) \kappa > 1$ as $n>1$ and $\kappa \geq 1/2$ (the factor $1/2$ in the definition of $J(n)$ being attained when considering a ball).
 As $0^n \in \bx(\tau)$, so is the convex combination
 $$
 c_\tau = \frac{1}{(2n-1) \kappa} c'_\tau + \left(1-\frac{1}{(2n-1) \kappa} \right)0^n. 
 $$
 
 Choose any $y\in \tau$. Then by the triangle inequality,
 $$
 d(y, c_\tau) \leq \frac{1}{(2n-1) \kappa} d(y, c'_\tau)+ \left(1-\frac{1}{(2n-1) \kappa}\right) d(y, 0^n) \leq 
 $$
 $$
 \leq 1 + \Big((2n-1)  - 1/\kappa\Big) = (2n-1) \left( \frac{1}{2n-1} + 1 - \frac{1}{(2n-1) \kappa}\right)=
 $$
 $$
 =(2n-1) \left(  1 - \frac{1-\kappa}{(2n-1) \kappa}\right) \leq (2n-1) -\frac{1-\kappa}{\kappa} = (2n-1) - \rho.
%= (2n-1) J(n) = (2n-1) \Big(1 - (1-J(n))\Big).
 $$
 and similarly, $d(0^n, c_\tau) \leq 1$, which implies $c_\tau \in ([-1,1]^{n-1} \times [0,1])$. Thus the theorem holds.
\end{proof}

For the purposes of the inductive proof of Theorem \ref{Thm:Main1} we will require the $\lec(n, \rho(n))$ property with a parameter $\rho(n)$ which is monotone in $n$. Combining with (1) of Theorem \ref{Thm:JungBounds} we thus phrase the following.

\begin{corollary}
\label{Cor:Main} 
Fix an integer $n > 1$. Then $\lec(n,\rho(n))$ holds for $\rho(n) = \frac{1}{n}$.
\end{corollary}

\begin{proof}
 Apply $\kappa = \frac{n}{n+1}$ from (1) of Theorem \ref{Thm:JungBounds} to Lemma \ref{Lem:Main}.
\end{proof}

%---------------------------------------------------------------------------
%%%%%%%%%%%%%%%%%%%%%%%%%%%%%
\section{Contractibility of integer lattices}
\label{Sect:MainContract}

In Subsection \ref{Sub:1} we use the $\lec$ property inductively to conclude the contractibility of Rips complexes of integer lattices at large scales. In Subsection \ref{Sub:2} we provide ad-hoc local crushings in dimension $3$, which yield smaller scales of contractibility for Rips complexes of $\ZZ^3$.

In both of these arguments we will be using the \textbf{lexicographical order} in $\RR^n$ defined as follows: $(x_1, x_2, \ldots, x_n) \prec (y_1, y_2, \ldots, y_n)$ iff the largest index $i$ at which the two vectors differ satisfies $x_i < y_i$. When $x_i, y_i \in \{0, 1, \ldots, 9\}$, the equivalent condition is that the corresponding $n$-digit numbers satisfy $x_{n} x_{n-1}\ldots x_2 x_1 < y_{n} y_{n-1}\ldots y_2 y_1.$ In comparison to the standard lexicographical order we have reversed the indexing so as to have the last coordinate be the dominant one. Also, for a positive real number $w$ let $(w\ZZ)^n$ denote the integer lattice scaled by $w$, i.e., $ \{ w \cdot z \mid z \in \ZZ^n\}.$ For each $i$ let $\pi_i \colon \RR^n \to \RR$ denote the projection to the $i^{\textrm{th}}$ coordinate.

%-------------------------------
%%%%%%%%%%%%%
\subsection{Contractibility of integer lattices at large scales}
\label{Sub:1}

\begin{theorem}
 \label{Thm:Main1}
For each $r \geq n^2 (2n-1)$ we have $\Rips((\ZZ^n, d_1), r) \simeq \bullet$.
\end{theorem}

\begin{proof}
The initial two cases for $n=1, 2$ follow from Theorem \ref{Thm:Rips12}.

Assume $n> 2$, define $\rho = \frac{1}{n}$ and observe that $\lec(n', \rho)$ holds for all $n' \in \{2, 3, \ldots, n\}$ by Corollary \ref{Cor:Main}. 
 Choose any real $m \geq  n / \rho $. We \textbf{claim} that $\Rips((1/m \cdot \ZZ)^n, (2n-1))$ is contractible. Deferring the proof of the claim for a moment, we observe that the claim combined with the scaling of $(1/m \cdot \ZZ)^n$ by $m$ implies that $\Rips(\ZZ^n, (2n-1) m)$ is contractible, for each $m \geq n / \rho$. This means that the theorem holds for all scales $r = (2n-1)m \geq n^2 (2n-1)$.

It thus remains to prove the claim that $\Rips((1/m \cdot \ZZ)^n, 2n-1)$ is contractible. By the Whitehead theorem it suffices to show that for each sufficiently large $M>0$, the Rips complex $\Rips\left(X, 2n -1\right)$ is contractible for $X=(1/m \cdot \ZZ)^n \cap [-M, M]^n$, as each element of a homotopy group of $\Rips((1/m \cdot \ZZ)^n, 2n-1)$  is contained in a $\Rips(X, 2n-1)$ for a large enough $M$. 

The proof that $\Rips(X, 2n-1)$ is contractible proceeds by induction. Let $x_1, x_2, \ldots, x_\ell$ denote the collection of points of $X = X_0$ arranged in the lexicographical order and define $X_i = X \setminus \{x_1, x_2, \ldots, x_i\}$. We will prove that for each $i \in \{0,1, 2, \ldots, \ell -2\}$, $x_{i+1}$ is locally crushable in $X_i$ at scale $2n-1$. By Theorem \ref{Thm:Crush}, this will imply that $\Rips(X, 2n-1) \simeq \Rips(X_{\ell -1}, 2n-1)$ and is thus contractible as $X_{\ell-1}$ is a single point.

There are two different types of inductive steps to prove that $x_{i+1}$ is locally crushable in $X_i$ at scale $2n-1$.

\begin{enumerate}
 \item The first type is used when the last coordinate of $x_{i+1}$ (namely, $\pi_n(x_{i+1})$) is not maximal within $X$, i.e., $\pi_n(x_{i+1}) < \max_{z\in X_i} \pi_n(z)$. By the definition of the lexicographical order, $X_i$ contains the entire subgrid of $X$ above $x_{i+1}$, namely the non-empty subset
 $$
X'_i=(1/m \cdot \ZZ)^{n} \cap ([-M, M]^{n-1} \times \left[ \pi_n(x_{i+1}) + 1/m,  M \right]).
 $$
 We will now prove that $x_{i+1}$ is locally crushable in $X_i$ with the subset $L_{x_{i+1}}$ (see Definition \ref{Def:LocCrush}) being 
 $$
 L_{x_{i+1}}=\{z \in X_i \mid 1 \geq |\pi_j(z)-\pi_j(x_{i+1})|, \forall j \in \{1, 2, \ldots, n\}; \pi_n(z) > \pi_n(x_{i+1})\}.
 $$
 Observe that $ L_{x_{i+1}}$ is contained in the box $[-1,1]^{n-1} \times [0,1]$ translated by $x_{i+1}$. We refer to this translated box as $B_i$. 
 
 Following the definition of the local crushing, choose a simplex $\tau \in \Rips(X_i, 2n-1)$ containing $x_{i+1}$. By the $\lec$ property there exists $c_\tau \in B_i \cap \bx(\tau): d(y, c_\tau) \leq 2n-1 - \rho$, for each $y \in \tau$. Next, we change each of the coordinate of $c_\tau$ by at most $1/m$ to reach a point $c'_\tau$ of $L_{x_{i+1}}$:
\begin{itemize}
 \item For $j < n$ we snap the $j^{\textrm{th}}$ coordinate of $c_\tau$ (namely, $\pi_j(c_\tau)$) to the closest value of $1/m \cdot \ZZ$ in the direction of $\pi_j(x_{i+1})$. (If $\pi_j(c_\tau)=\pi_j(x_{i+1})$, no change occurs.)
\item We snap the $n^{\textrm{th}}$ coordinate to the closest value of $1/m \cdot \ZZ$ in the direction of $\pi_j(x_{i+1})$. In case that new value is $\pi_j(x_{i+1})$, we have $\pi_n(c_\tau) \in [\pi_j(x_{i+1}), \pi_j(x_{i+1}) + 1/m)$ and we thus snap to $\pi_j(x_{i+1}) + 1/m$ instead.
\end{itemize}
The resulting point is denoted by $c'_\tau$. We first argue that $c'_\tau \in X_i$. By the $\lec$ property, $c_\tau \in \bx(X)$. By the structure of the $X$ (the cubical grid), the above modification $c_\tau \to c'_\tau$ snapping coordinates towards $x_{i+1} \in X$  (or the last coordinate to $\pi_n(x_{i+1}) + 1/m$) thus results in $c'_\tau$ being in $X$. Since $\pi_n(c'_\tau) >  \pi_n(x_{i+1})$, and since $X_i$ contains all the points of $X$ with $n^{\textrm{th}}$ coordinate above $\pi_n(x_{i+1})$ due to the lexicographical order (i.e., it contains $X'_i$), we have $c'_\tau \in X_i$.

Next we argue that the resulting point $c'_\tau$ is in $L_{x_{i+1}}$:
\begin{itemize}
 \item For $j < n$, as $|\pi_j(c_\tau) - \pi_j(x_{i+1})| \leq 1$ due to the $\lec$ property, so is  $|\pi_j(c'_\tau) - \pi_j(x_{i+1})| \leq 1$ as the $j^{\textrm{th}}$ coordinate was moved (if moved) towards $\pi_j(x_{i+1})$. 
 \item In the same spirit we can observe that $\pi_n(c'_\tau) - \pi_n(x_{i+1}) \in (0, 1]$. In particular, this means that $c'_\tau \neq x_i$.
\end{itemize}
Thus $c'_\tau$ is in $L_{x_{i+1}}$. 

Changing each coordinate by at most $1/m$ in the transition $c_\tau \to c'_\tau$ provides a change by at most $n/m$ in the distance. Thus for each $y\in \tau$: 
$$
d(y, c'_\tau) \leq d(y, c_\tau) + d(c_\tau, c'_\tau) \leq (2n-1 - \rho) +n/m \leq 2n -1
$$ 
by the definition of $m$. The locally crushing condition thus holds, with $c'_\tau$ being the point corresponding to $\tau$ according to Definition \ref{Def:LocCrush}. 

\item The second type of the inductive step on the index $i$ uses another induction, that on the dimension $n$. Let $i$ be the first index, for which $\pi_n(X_i)$ is a single point. This means that 
$$
X_i = \left((1/m \cdot \ZZ)^{n-1} \cap [-M, M]^{n-1}\right) \times \{\pi_n(x_{i+1})\}.
$$
By the statement of this theorem for $n-1$, we may inductively assume that a sequence of local crushings induces the homotopy equivalence $\Rips(X_i, 2n-1) \simeq \Rips(X_{\ell-1}, 2n-1)$, because the relevant parameters have been chosen so that they hold for the case $n-1$ as well:
\begin{itemize}
\item We have chosen $\rho$ so that $\lec(n', \rho)$ holds for all $n' \in \{2, 3, \ldots, n\}$.
 \item We have chosen $r \geq  n^2 (2n-1)$ and thus $r \geq  (n-1)^2 (2(n-1)-1)$ also holds.
 \item We have chosen $m \geq n/\rho$ and thus $m \geq (n-1)/\rho$ also holds.
\end{itemize}
\end{enumerate}
Thus $X$ can be transformed to $X_{\ell-1}$ by a sequence of local crushings and thus $\Rips(X, 2n-1) \simeq \Rips(X_{\ell -1}, 2n-1) \simeq \bullet$.
The only exception in this inductive step on dimension is the initial case. If $n-1=2$, then a sequence of local crushings (in fact, they are crushings in this case) inducing the homotopy equivalence $\Rips(X_i, 2n-1) \simeq \Rips(X_{\ell-1}, 2n-1)$ has been provided in the proof of  Theorem \ref{Thm:Rips12}.
\end{proof}

%-------------------------------
%%%%%%%%%%%%%
\subsection{Contractibility of integer lattices at small scales}
\label{Sub:2}

The overall strategy of the next result is the same as that of Theorem \ref{Thm:Rips12} (see Figures \ref{Fig1} and \ref{Fig2}) with the incorporation of local crushings. It is also an adaptation of the strategy of Theorem \ref{Thm:Main2}, using explicit local domination steps instead of the ones implied by the $\lec$ property.

\begin{theorem}
 \label{Thm:Main2}
 For each $r \geq 3$, $\Rips(\ZZ^3, r) \simeq \bullet$.
\end{theorem}

\begin{proof}
As in the proof of Theorems \ref{Thm:Rips12} and \ref{Thm:Main2}, it suffices to show that for arbitrarily large $M>0$ and $Z=\{0, 1, \ldots, M\}^3$, we have $\Rips(Z,r) \simeq \bullet.$ Again, we will proceed by induction, removing points from $Z$ in the lexicographical order. Assume $z_1, z_2, \ldots, z_\ell$ are the points of $Z$ ordered in the lexicographical order, and define $Z_i = Z \setminus \{z_1, z_2, \ldots, z_i\}$. We will prove that for each $i \in \{0,1, 2, \ldots, \ell-2\}$, $z_{i+1}$ is locally crushable in $Z_i$, yielding $\Rips(Z,r) \simeq \Rips(Z_{\ell-1},r)=\bullet.$

Fix $i$. Translating by $-z_{i+1}$ we recenter $Z_i$ so that $z_{i+1}$ corresponds to $0^3$. Thus, there exist integers $a_1, a_2 \leq 0$ and  $b_1, b_2, b_3 \geq 0$ so that $Z_i$ consists of those points in 
$$
\{a_1, a_1 +1, \ldots, b_1\} \times  \{a_2, a_2 +1, \ldots, b_2\} \times \{0, 1, \ldots, b_3\}  
$$
which appear lexicographically after $0^3$. We proceed by case analysis:
\begin{enumerate}
 \item If $i$ is the first index at which $b_3=0$, the grid $Z_i$ is two-dimensional and thus there is a sequence of local crushings inducing $\Rips(Z_i,r)\simeq \bullet$ by Theorem \ref{Thm:Rips12}.
 \item If $b_3 \neq 0$ but $b_1=b_2 = 0$, then all the points of $Z_i$ except for $0^3$ have the last coordinate at least $1$, hence $0^3$ is dominated in $\Rips(Z_i,3)$ by $(0,0,1)$. The technical reason is that for each $w \in Z_i, w \neq 0^3$ we have $d(w,(0,0,1)) = d(w,0^3)-1$.
 \item Similarly, if $b_2, b_3 \neq 0$ but $b_1=0$, then $0^3$ is dominated in $\Rips(Z_i,3)$ by $(0,1,1)$. The technical reason is that for each $w \in Z_i$ we have $d(w,(0,1,1)) \leq d(w,0^3)$. This is easy to see because any point in $Z_i$ other than $0^3$ has either the second or the third coordinate non-zero and is thus at distance at most $r-1$ from either $(0,1,0)$ or $(0,0,1)$. The same argument was used in the proof of Theorem \ref{Thm:Rips12}, case $n=2$, item (1).
 Analogously,  if $b_1, b_3 \neq 0$ but  $b_2=0$, then $0^3$ is dominated in $\Rips(Z_i,3)$ by $(1,0,1)$.
 \item Assume $b_1, b_2, b_3 \neq 0$. We will prove that the local crushing condition holds for $0^3$ with $L_{0^3}=\{(0,0,1), (1,1,0), (-1,1,0)\},$ which is of diameter $3$. Assume $\tau \in \Rips(Z_i, r)$ contains $0^3$. If $\tau \cup \{(0,0,1)\}$ is not a simplex, then there exists $(t_1, t_2,0)\in \tau: |t_1| + t_2 = r$ (note that $t_2  \geq 0$ due to the lexicographical order). 
\begin{itemize}
	\item First assume $t_1 \geq 0$. Take any $(s_1, s_2, s_3)\in \tau$ different than $0^3$. As $d((t_1, t_2,0), (s_1, s_2, s_3)) \leq r$, we can't have both $s_1$ and $s_2$ equal to $0$. In particular, either $s_2> 0$ or $s_1 > 0$. This means that $(s_1, s_2, s_3)$ is either at distance at most $r-1$ from $(0,1,0)$ or $(1,0,0)$ respectively. In both cases we can conclude $d((s_1, s_2, s_3), (1,1,0)) \leq r$ as in (3) above. This implies $\tau \cup \{(1,1,0)\}\in \Rips(Z_i, r)$.
	\item The case where $t_1 \leq 0$ is handled analogously with the conclusion that $\tau \cup \{(-1,1,0)\}\in \Rips(Z_i, r)$.
\end{itemize} 
   These two cases conclude the proof of local crushing.
\end{enumerate}
\end{proof}

%---------------------------------------------------------------------------
%%%%%%%%%%%%%%%%%%%%%%%%%%%%%
\section{Conclusion}
\label{Sect:Conclusion}

We conclude with two comments on our results. 

First, \textbf{the proof of Theorem \ref{Thm:Main1} actually holds for any $d_p$ metric}. For finite $p \geq 1$, $d_p$ is defined as
$$
d_p\left( (a_1, a_2, \ldots, a_n), (b_1, b_2, \ldots, b_n) \right) = \left(\sum_{i=1}^n |a_i - b_i|^{p}\right)^{1/p},
$$
while 
$$
d_\infty\left( (a_1, a_2, \ldots, a_n), (b_1, b_2, \ldots, b_n) \right) = \max_{i=1, 2, \ldots, n} |a_i - b_i|.
$$
There are only few places in the proof of Theorem \ref{Thm:Main1} that actually use specifics of $d_1$. We next discuss them and comment on why the same argument holds for any $d_p$ metric with $p \in [1,\infty]$. 
\begin{enumerate}
 \item The proof makes use of $ L_{x_{i+1}}$, which is contained in the box $[-1,1]^{n-1} \times [0,1]$. The diameter of this box is $2n-1$ in $d_1$ (see Remark \ref{Rem:Box}) and strictly less in other $d_p$ metrics. Hence the estimates of the proof hold for any other $d_p$ metric. 
 \item Another property that is used is the fact that the center of the minimal ball in $(\RR^n, d_1)$ containing $A$ lies in $\bx(A)$. The same argument that is given just before Definition \ref{Def:Jung} implies that the same holds in any $d_p$ for $p\in [1,\infty)$. A special case is $d_\infty$: there a center is not unique, but it is easy to see (again, by the same argument that is given just before Definition \ref{Def:Jung}) that it can be chosen in $\bx(A)$.
 \item All the other metric estimates in the proof hold for any $d_p$ as they only use the triangle inequality.
 \item The bound of (1) of Theorem \ref{Thm:JungBounds} holds for ant $d_p$ metric by the same references (\cite{Bohn}, \cite[Proposition 2.12]{Amir}) and consequently so does the auxiliary Lemma \ref{Lem:Main}.\end{enumerate}
 The smaller diameter of the mentioned box in (1), and tighter bounds on the Jung constant  in (4) can be used to deduce better bounds of the main result for $d_p$ metrics. 
 %In the light of the question below, we refrain from pursuing this direction. 
 
 The second comment is related to the following \textbf{question: Is the Rips complex } $\Rips((\ZZ^n, d_1),r)$ \textbf{ contractible for each} $r \geq n$? We conjecture that the same proof strategy as that in the proof of Theorem \ref{Thm:Main2} should work. Unfortunately, we failed to noticed a general pattern in the case analysis of the said proof, that would extend to higher dimensions. However, it is clear that for each fixed dimension $n$, such a case analysis only needs to consider finitely many cases, hence specific low dimensions may be amenable to computational verification. Concerning the lower bound for the scale parameter, it is easy to see that for $r<n$, the complex $\Rips((\ZZ^n, d_1),r)$ is not contractible: according to \cite{AVCubes} the Rips complex  $\Rips\left((\{0,1\}^n, d_1),r\right)$ has non-trivial homology, and since there is a contraction ($1$-Lipschitz retraction) $\ZZ^n \to \{0,1\}^n$, the inclusion $ \{0,1\}^n \hookrightarrow \ZZ^n$ induces an injection on the homology of Rips complexes at each scale $r$, see \cite{ZVContractions}.

%---------------------------------------------------------------------------
%%%%%%%%%%%%%%%%%%%%%%%%%%%%%
\section{Acknowledgments}

The author would like to thank Henry Adams, Arseniy Akopyan, and Matthew Zaremsky for fruitful discussions on the subject. 
The author would also like to thank the referee for a very careful reading and numerous helpful comments, which improved the disposition of the paper. The author would also like to thank Samir Shukla for pointing out a gap in the journal version of the proof of Theorem 5.2. The gap has been fixed in this version.
The author was supported by the Slovenian Research Agency grants No. J1-4001 and P1-0292.

%---------------------------------------------------------------------------
%%%%%%%%%%%%%%%%%%%%%%%%%%%%%

\end{document}